\documentclass[journal,twoside,web]{ieeecolor}
\let\ACMmaketitle=\maketitle
\renewcommand{\maketitle}{\begingroup\let\footnote=\thanks \ACMmaketitle\endgroup}
\usepackage{amsmath,amssymb,amsfonts,mathtools,generic}
\usepackage{blindtext,color}
\usepackage{latexsym}
\usepackage{graphicx}
\usepackage{subcaption}
\usepackage[mathscr]{euscript}

\def\BibTeX{{\rm B\kern-.05em{\sc i\kern-.025em b}\kern-.08em
    T\kern-.1667em\lower.7ex\hbox{E}\kern-.125emX}}
\newcommand{\R}{\mathcal{R}}
\newcommand{\G}{\mathcal{G}}
\newcommand{\F}{\mathcal{F}}
\newcommand{\V}{\mathcal{V}}
\newcommand{\E}{\mathcal{E}}

\newcommand{\N}{\mathcal{N}}

\newcommand{\SA}{\mathcal{S}}

\newcommand{\1}{\textbf{1}}

\newtheorem{theorem}{Theorem}[section]
\newtheorem{lemma}[theorem]{Lemma}
\newtheorem{corollary}[theorem]{Corollary}

\begin{document}

\title{Dynamic social learning under graph constraints}

\author{Konstantin Avrachenkov, Vivek S.\  Borkar, \IEEEmembership{Fellow, IEEE},
Sharayu Moharir, and Suhail Mohmad Shah
\thanks{\textit{Authors listed alphabetically}.}
\thanks{This work was supported by the grant `\textit{Machine Learning for Network Analytics}' from the Indo-French Centre for Promotion of Advanced Scientific Research.\ The work of VB was also supported in part by a J.\  C.\  Bose Fellowship from the Government of India. The work of KA was also supported in part
by grant ``Distributed Learning and Control for Network Analysis'' from Nokia Bell Labs.}
\thanks{Konstantin Avrachenkov is with INRIA Sophia Antipolis, 2004 Route des Lucioles,
Valbonne 06902, France (e-mail: K.Avrachenkov@inria.fr).\ }
\thanks{VB and SM are and SMS was with
the Department of Electrical Engineering, Indian Institute of Technology Bombay,
Mumbai 400076, India. SMS is now with the Department of Electrical Communications Engineering, Hong Kong Uni.\ of Science and Technology, Clear Water Bay, Kowloon, Hong Kong.  (e-mail: borkar.vs@gmail.com; sharayu.moharir@gmail.com; suhailshah2005@gmail.com).}
}

\maketitle

\begin{abstract}

We introduce a model of  graph-constrained  dynamic choice with reinforcement modeled by positively $\alpha$-homogeneous rewards. We show that its empirical process, which can be written as a stochastic approximation recursion with Markov noise,  has the same probability law as a certain vertex reinforced random walk.  We use this equivalence to show that for $\alpha > 0$, the asymptotic outcome concentrates around the optimum in a certain limiting sense when `annealed' by letting $\alpha\uparrow\infty$ slowly.\

\end{abstract}


\section{Introduction}
\IEEEPARstart{D}{ynamic} choice models, wherein the subsequent choice of one among finitely many alternatives depends upon the relative frequency with which it has been selected in past, have found many applications.\  This is so particularly in the scenario when the higher the frequency, the higher the probability of an alternative being chosen again.\ Such `positive reinforcement' is seen in models of herding behavior \cite{Chamley}, evolution of conventions \cite{Young1}, `increasing returns' economics \cite{Arthur}, etc.\   Similar dynamics also arise in other disciplines, e.g., population algorithms for optimization \cite{BorkarDas} and more recently, for service requests in web based platforms for search, e-commerce, etc.\ \cite{Shah}.\
One common caveat in all these is what is already the concern of the aforementioned  models of herding and increasing returns economics, viz., the risk of some initial randomness leading to the process getting eventually trapped in an undesirable or suboptimal equilibrium behavior.\  In this work we present a different take on this issue.\ Firstly, we introduce what we call a graph-constrained framework, wherein the choice at any instant is restricted by the choice during the previous instant.\ This is a realistic scenario that reduces to the classical case when the constraint graph is fully connected. Some examples are:\\

\noindent 1.\ Consider buyers buying a product on an e-commerce portal.\ They are influenced by both the average rating (assumed to be  stable) and the number of people who bought the product, as reflected in the number of reviews.
In this application the graph constraints come from
suggestions from the e-commerce portal for purchase of items from the same or related categories.\\

\noindent 2.\ Consider the task of locating an object in a large image using crowdsourced agents.\ Typically, the image is split into multiple sub-images and each agent is asked to examine a few sub-images for the desired object.\ Since the image is large, it is desirable to determine which sub-image to examine next based on partial information of the current state.\ One way to do this is to constrain the next sub-image to be one of the neighbors of the sub-image examined most recently, chosen randomly according to a probability distribution based on the current information from the crowd about these sub-images.\ See, e.g., \cite{Dempsey} for one potential real application.\\

\noindent 3.\ A graphical constraint may also arise in a scenario where a mobile sensing unit (e.g., a robot or a UAV) covers an area repeatedly.\  It has to plan its trajectory according to certain objectives which prioritize dynamically the preferred regions or `hot spots'.\ The movement, however, can only be to neighboring positions.\ If there is no central coordinator, then one is faced with the kind of problem we have.\\

\noindent 4.\  Online video sharing platforms such as YouTube make yet another application case.
Typically, after a user has seen a video, he or she is recommended a list of suggested videos.
The videos are recommended based on semantic similarity and the number of views. In this case,
the graphical constraints come from the physical limitation of the screen (typically no one scrolls
down more than one or two screens) and semantic similarities. Furthermore, the system is more likely
to recommend a content with a large number of views and the user is also more likely to click on a
content with a significant number of views. Our model not only confirms that this leads to the effect
of social bubbles \cite{Pariser2011}, but also proposes a way of tuning the recommendation mechanism to break
such bubbles.\\

As indicated above, optimality is not guaranteed in many of the aforementioned models because of the dynamics getting trapped in a suboptimal limit, the so called `trapping' phenomenon \cite{Arthur}. We show here that by suitably tuning or `annealing' the choice probabilities, the asymptotic profile can be made to concentrate on the optimal behavior.\ The tuning scheme increases the concentration of probability on the current front runner and corresponds to the natural phenomenon whereby the agents' confidence in their choices increases with increasing adoption thereof by their peers. Our agents are autonomous, though influenced by the past. Thus the final outcome is \textit{emergent} and not \textit{engineered}. In the basic model (i.e., without the aforementioned `annealing'), we get convergence to a common decision, but not necessarily to an optimal one. The  `annealed' variant on the other hand ensures the latter, i.e., asymptotic optimality. It should be emphasized that while we borrow terminology from simulated annealing (SA), our annealing scheme modulates the net drift, i.e., the driving vector field of the stochastic approximation iteration and \textit{ipso facto} its limiting o.d.e., in order to achieve optimality, its effect on the noise component is unimportant. This is unlike classical SA where it is the noise variance that is being tuned. We do add extraneous noise to the choice probabilities (see (\ref{transprob}) below) just as in SA, but its aim is to ensure that unstable equilibria are avoided almost surely, not to ensure avoidance of \textit{stable} suboptimal equilibria as in SA. The former is an easier objective as it entails only some `persistent excitation' (to borrow a phrase from control theory) to push the iterates away from unstable equillibria and their stable manifolds, and does not call for `hill climbing' with noise as in SA. The slow morphing of the drift is tantamount to morphing of the landscape itself to make it more `peaked' while retaining the same optima. (That is, ratio of the function value at a global maximum to that at a local maximum which is not a global maximum progressively increases, but their locations don't change.)
The dynamics in question is closely related to similar dynamics arising in connection with vertex reinforced random walks \cite{Benaim}, a fact we exploit.

We give brief comparisons with some  related works in multiarmed bandits in order to highlight the differences.\ In \cite{Shah}, a related  model is considered and it is observed that the process may get locked into suboptimal equilibria.\ The remedy they propose is to randomize the rewards for a fixed time window in a clever manner (dubbed a `balanced' exploration) before the aforementioned dynamic choice process takes over.\ We eschew any such modification and instead take recourse to the above scheme which  is indeed optimal in the  limit.\ This result is of a distinct flavor compared to \cite{Shah}.\ Also, our techniques are different, as are our objectives: we seek asymptotic optimality and do not consider regret.\ In \cite{Cohen}, which is methodologically closer to our work, a full fledged game problem is considered wherein many agents are concurrently exercising their choices with their payoffs depending on others' choices as well.\ Their focus is on $\varepsilon$-Nash equilibria and not on optimal behavior  as in our (non--game theoretic) work.\ While the core technique, viz., use of the multiplicative weight rule, is common between this work and \cite{Cohen}, they use a different choice thereof.\ Graphical constraints analogous to ours are used in \cite{Sankar} in a bandit framework, but they are motivated by how communication among agents can be factored into the analysis. In general, bandit algorithms do not involve graphical constraints and their focus is on non-asymptotic behavior unlike ours. However,  graphical restrictions in bandit context do arise in a number of practical applications
and have  important implications. The standard algorithms deployed to solve bandit problems such as the $\epsilon$-greedy strategy or \textit{UCB} algorithm \cite{Lattimore} may fail to  achieve optimal behavior under graph constraints, as one may get stuck with a choice with a sub-optimal reward.  We substantiate this claim in Section \ref{Simulation} with a simple example.

We draw upon the framework of \cite{Benaim} substantially.\ (See \cite{Benaim2, BenaimRaimond,  BenaimRS} for  extensions.) The key contribution of \textit{ibid.} is the analysis of a general vertex reinforced random walk using the `o.d.e.' approach to stochastic approximation. It derives very broad results about their asymptotic behavior, and then narrows these down to concrete examples with linear reinforcement to obtain stronger claims.\  Our model is pitched in between - it is a nonlinear model, but a very specific one and allows for more specific claims to be established.\  Use of annealing ideas in this context is another novelty of our work.


Such graphically constrained choice models can also be posed as stochastic combinatorial optimization problems. A well known heuristic for solving such problems is simulated annealing. However, SA \textit{with noisy observations} is well known to be sample inefficient  \cite{Bouttier}, \cite{Gelfand-1}, \cite{Gutjahr}. In fact, the best possible sample complexity results that have been obtained (Theorem 3, \cite{Bouttier}) require that the number of samples required per iteration increase to infinity with the iteration count.  This
\begin{center}
\begin{tabular}{ |c|c| }
 \hline
\multicolumn{2}{|c|}{\textbf{Key Notation}}\\
 \hline
 $\mu_i$ & Reward associated with object $i$. \\
 $m$ & Number of objects. \\
 $S_i(n)$ & Number of times $i$ was picked. \\
 $x_i(n)$ & Relative frequency  $S_i(n)/n$. \\
$\mathcal{S}_m$ & Unit simplex in $\mathbb{R}^m$.\\
$\text{int}(\mathcal{S}_m)$ & Interior of $\mathcal{S}_m$. \\
$\mathcal{G}$ & Directed graph.\\
$\mathcal{V}$ & Node set of $\mathcal{G}$.\\
$\mathcal{E}$ & Edge set of $\mathcal{G}$.\\
$\mathcal{N}(i)$ & Neighbourhood of $i$.\\
$\zeta(n)$ & Noise in reward vector.\\
$\mathcal{F}_n$ &  $\sigma\big(\xi(k), \zeta_i(k), 1 \leq i \leq m, k \leq n\big)$.\\
$\hat{\mu}_i(n)$ & Empirical mean, see (\ref{slln}).\\
$\epsilon(n)$ & Exploration time step (see (\ref{epreduce})).\\
$c(n),a(n)$ & See (\ref{epreduce}) and (\ref{varep}).\\
$f_i^\alpha(x)$ & Reinforcement function, $(\mu_i x_i)^\alpha$.\\
$\alpha$ & Reinforcement exponent, see above.\\
$\chi_\cdot(i)$ & Uniform distribution on $\mathcal{N}(i)$.\\
$p_{ij}^{\alpha}(x)$ & Transition prob. of $\{\xi(n)\}$, see (\ref{probdefn2}).\\
$\pi^\alpha(x)$ & Stationary distribution of $p_{ij}^{\alpha}(x)$.\\
$\varphi^\alpha(x)$ & See (\ref{scaled-rep}).\\
$\iota_i(n)$ & See (\ref{iota}). \\
$A$ & Adjacency matrix, $A:= [[a_{ij}]]_{i,j \in \mathcal{V}}$.\\
$T$ & Temperature, defined as $1/\alpha$.\\
$b(n)$ & Time step in $T$, see (\ref{cooling}).\\
$D$ & $\{ i \in \mathcal{V} \, : \, \mu_i = \max_j \mu_j \}$.\\
$f(n) = O(g(n))$ & $\limsup_{n  \to \infty} \frac{|f(n)|}{g(n)} <\infty$.\\
$f(n) = \Omega(g(n))$ & $g(n)= O(f(n))$. \\
$f(n) = o(g(n))$ & $\lim_{n \to \infty} \frac{|f(n)|}{g(n)} = 0$.\\
$f(n) = \omega(g(n))$ & $g(n)= o(f(n))$.\\
$f(n) = \Theta(g(n))$ & $f(n)= O(g(n)) \text{ and }g(n)= O(f(n))$.\\
 \hline
\end{tabular}
\end{center}
makes deploying SA with noisy observations quite difficult, particularly for applications where obtaining samples may entail time consuming simulations. In contrast, our algorithm needs one sample per iteration under i.i.d.\ bounded variance noise, which makes it much more sample efficient as compared to SA with noisy observations.

We describe our model in the next section and demonstrate its connection with the vertex reinforced random walk.
Section 3 provides convergence analysis of the basic scheme. In section 4 we analyze its `annealed' counterpart, leading to the desired result.\ Section 5 specializes the problem to complete graph where we can say more.\ Section 6 provides some numerical experiments.\ Three appendices sketch some  technical issues left out of the main text for ease of reading.

\textbf{Notation:} For ease of reference, we list the key notation used in the paper in the above table. This includes the standard Big-O notation used throughout the paper.

\section{Problem formulation}

In this section we set up our model of choice dynamics and the key notation. \\

\noindent \textbf{Model:} Consider a stream of agents arriving one at a time\footnote{This is for convenience. The identity of agents is irrelevant here and they may repeat as long as the choice mechanism remains the same.} and choosing one of $m > 1$ distinct objects, with a reward $\mu_i > 0$ associated with the $i$th object.  The $(n+1)$-st agent picks the $j$th object with conditional probability (conditioned on past history) $p_j(n)$, which we shall soon specify.\ Let $\xi(n) = i$ if the $n$th agent picks object $i$.\ Let $S_i(n) :=$ the number of times object $i$ was picked till time $n$ and $x_i(n) := \frac{S_i(n)}{n}, \ n \geq 1$, its relative frequency.\ Then  a simple calculation leads to the recursion
\begin{equation}
x_i(n+1) = x_i(n) + \frac{1}{n+1}\left(\mathbb{I}\{\xi(n+1) = i\} - x_i(n)\right), \ n \geq 0.\ \label{SA}
\end{equation}
Here $\mathbb{I}\{ \cdots \}$ is the `indicator function' which is $1$ if its argument is true and $0$ otherwise.\ For specificity, we arbitrarily set $x_i(0) = \frac{1}{m} \ \forall i$, suggestive of a uniform prior.
This will not affect our conclusions.\ Throughout, we use the convention $\frac{0}{0} = 0$.
The vector $x(n) := [x_1(n), \cdots , x_m(n)]^T$ takes values in the simplex of probability vectors,
$$\SA_m := \left\{x = [x_1, \cdots , x_m]^T : x_i \geq 0 \ \forall i, \ \sum_jx_j = 1\right\}.$$
We shall denote by int$(\SA_m)$ the interior of $\SA_m$. We assume that the observed reward at time $n$ for choice $i$ is not $\mu_i$, but
$\tilde{\mu}_i(n) = \mu_i +\zeta_i(n)$ where $\{\zeta_i(n), n \geq 0\}$ is i.i.d.\  zero mean noise with bounded variance.\\

 \noindent \textbf{Graphical Constraints:} We assume that the choice in the $(n+1)$-st time slot is constrained by the choice made in the $n$th slot, e.g., when, given the present choice, only some selected `nearby' or `related' choices are offered or preferred (see  examples in the introduction).\ We model this as follows.\ Consider a directed graph $\G =(\V, \E)$ where $\V, \E$ are resp., its node and edge sets, with $|\V| = m$.\ Assume that $\G$ is irreducible, i.e., there is a directed path from any node to any other node.\  Let $\N(i) := \{j \in \V : (i,j) \in \E\}$ denote the set of successors of $i$ in $\G$.  If  $i$ is chosen at any instant $n$, the next choice must come from $\N(i)$.\ We assume:\\

\noindent \textbf{(A1)} \textit{For each $i$,  $i \in \N(i)$.\ This implies a self-loop at each node, i.e., $(i,i) \in \E \ \forall \ i \in \V$.\ (Thus, in particular, $|\N(i)| \geq 2 \ \forall i$.) We also assume that the neighborhood structure is bidirectional, i.e., $i \in \N(j) \Longleftrightarrow j \in \N(i)$.}\\

 \noindent \textbf{Selection Policy:} Let $\F_n :=$ the $\sigma$-field $\sigma(\xi(t), \zeta_i(t), 1 \leq i \leq m, t \leq n)$.\ Then the vector process $x(n) \in \SA_m$, whose $i$'\textit{th} component  $x_i(n) := \frac{S_i(n)}{n}$, is assumed to satisfy (\ref{SA}) with
\begin{equation}
\mathbb{P} (\xi(n+1) = j | \F_n) = (1 - \varepsilon(n))\tilde{p}^\alpha_{\xi(n) j}(x(n)) + \varepsilon(n)\chi_j(\xi(n)).  \label{transprob}
\end{equation}
Here:
\begin{itemize}
\item
\begin{equation}
\tilde{p}_{ij}^\alpha(x) := \mathbb{I} \big\{j \in \N(i)\big\} \frac{\hat{f}_j^{\alpha,n}(x)}{\sum_{l \in\N(i)} \hat{f}_l^{\alpha,n}(x)} \ ,\ \label{probdefn}
\end{equation}
for $\hat{f}^{\alpha,n}_i(x) := (\hat{\mu}_i(n)x_i(n))^\alpha$, where
$$
\hat{\mu}_i(n) := \frac{\sum_{k=0}^n\mathbb{I}\{\xi(k) = i\}\tilde{\mu}_i(k)}{\sum_{k=0}^n\mathbb{I}\{\xi(k) = i\}}
$$
is the empirical estimate of $\mu_i$ at time $n$  recursively computed by
\begin{eqnarray}
\hat{\mu}_i(n+1)&=&
\Big(1- \frac{1}{S_i(n+1)} \Big)\hat{\mu}_i(n) + \frac{\tilde{\mu}_i(n+1)}{S_i(n+1)},\nonumber \\
&& \ \nonumber \\
&& \ \ \ \ \ \ \ \ \ \ \ \ \ \ \ \ \ \ \ \ \text{if } \xi(n+1)=i, \nonumber\\
&=& \hat{\mu}_i(n),\,\,\, \ \ \ \ \ \ \ \ \ \ \ \ \text{otherwise}, \label{slln}
\end{eqnarray}
 with $\hat{\mu}_i(0) := 0$.

\item $\{\varepsilon(n)\}$ satisfy the recursion
\begin{equation}
\varepsilon(n+1) = (1 - c(n))\varepsilon(n), \label{epreduce}
\end{equation}
where $0 < c(n) \downarrow 0, \ \sum_nc(n) = \infty, \ nc(n) \stackrel{n\uparrow\infty}{\to} 0.$
The last condition implies that $\sum_nc(n)^2 < \infty$. We also assume that
for $a(n) := \frac{1}{n+1}$,
\begin{eqnarray}
&&\sum_n\varepsilon(n)^m = \infty, \ \sum_na(n)\varepsilon(n) = \infty, \label{varep0} \\
&&\varepsilon(n) = \omega\left(\frac{1}{\sqrt{n}}\right).\ \label{varep}
\end{eqnarray}
One example is $c(n) = \frac{1}{1 + (n+1)\log (n+1)}$, which results in
$\varepsilon(n) = \Theta\left(\frac{1}{\log n}\right)$, see Appendix III for details.

\item $\chi_\cdot(i)$ is the uniform distribution on $\N(i), i \in \V$.

\end{itemize}
That is, with probability $1 - \varepsilon(n)$, we pick $\xi(n+1) = j$ with probability $p^\alpha_{\xi(n) j}(x(n))$, and with probability $\varepsilon(n)$, we pick it uniformly from $\N(\xi(n))$.\  As $\alpha\downarrow 0$, the process approaches a simple random walk on the graph that picks a neighbor with equal probability.\ As $\alpha\uparrow \infty$, the process at $i$ will (asymptotically) pick the $j \in \N(i)$ for which $\mu_jx_j = \max_{k\in\N(i)}\mu_kx_k$, uniformly.
An immediate observation is the following, proved in Appendix I.

\medskip

\begin{lemma}\label{SLLN} $\hat{\mu}_i(n) \to \mu_i$ a.s.\ $\forall i$.\ \end{lemma}

\medskip

Thus a.s., $\lim_{n\uparrow\infty}\hat{f}_i^{\alpha,n}(x) = f^\alpha_i(x) := (\mu_ix_i)^\alpha \ \forall \ i,x,\alpha$ and
\begin{equation}
\lim_{n\uparrow\infty}\tilde{p}_{ij}^\alpha(x) = p_{ij}^\alpha(x) := \mathbb{I} \big\{j \in \N(i)\big\} \frac{f_j^\alpha(x)}{\sum_{l \in\N(i)} f_l^\alpha(x)} \ .\ \label{probdefn2}
\end{equation}
The functions $f^\alpha_i$ are monotone increasing, which captures the  `positive reinforcement', i.e., the fact that increased choice of a particular object $i$ increases its probability of being chosen in future, all else remaining the same. Each $f_j^\alpha$ is a locally  Lipschitz function  in int$(\SA_m)$, strictly increasing in  $x_j$ and satisfying  \textit{positive $\alpha$-homogeneity}: $f_j^\alpha(ax_j) = a^\alpha f_j^\alpha(x_j)$ for $a \geq 0$.
Then  $\mu_ix_i$ can be viewed as the fraction of the total reward accrued  by the fraction of population that chose $i$.\ Thus, e.g., in example 1  in the introduction, it is the average rating of $i$ times the fraction of the  customers who bought $i$ from among all who bought similar products.\ (In fact, it can be the \textit{number} thereof rather than the \textit{fraction}, because the normalization factor cancels out in the transition probability defined in (\ref{probdefn2}).) Its homogeneity property renders the choice probabilities  defined in (\ref{probdefn2}) scale-independent, as it should.\   Since our selection probability for $i$  will be proportional to $f_i^\alpha(x_i)$, a higher value of $\alpha$ makes the preference more peaked in the sense already described: it concentrates the probability mass further near global maxima, thereby  putting higher weight on `exploitation' than on `exploration'. Smaller $\alpha$ do the opposite. The `annealed' scheme we propose later slowly increases $\alpha$ to capture the trade-off  between the two.


\section{Convergence analysis}

This section  analyzes the convergence of the above scheme for fixed $\alpha$ using the theory of stochastic approximation \cite{BorkarBook}. The standard stochastic approximation algorithm is
\begin{equation}
y(n+1) = y(n) + a(n)[F(y(n), Y(n+1)) + \iota(n) + W(n+1) ] \label{SA2}
\end{equation}
where the possibly random positive stepsizes $\{a(n)\}$ satisfy $\sum_na(n) = \infty$, $\sum_na(n)^2 < \infty$, the `martingale noise' $\{W(n)\}$ satisfies $E[W(n+1)|\F'_n]=$ the zero vector for $\F'_n := \sigma(y(t),a(t),Y(t),\iota(t),W(t), t \leq n)$, $\iota(n) \to 0$ componentwise a.s., and the `Markov noise' $\{Y(n)\}$ satisfies
$P(Y(n+1) \in \cdot | \F'_n) = \hat{p}_{y(n)}( \cdot | Y(n))$
for a suitable transition probability $\hat{p}_y( \cdot | \cdot )$ parametrized by $y$. Then (\ref{SA}) has this form with $y(n) = x(n), a(n) = \frac{1}{n+1},$
\begin{eqnarray*}
W_i(n) &=& I\{\xi(n+1) = i\}   - (1 - \varepsilon(n))p^\alpha_{\xi(n)i}(x(n))  \\
&&- \ \varepsilon(n)I\{i \in \N(\xi(n))\}/m_i,
\end{eqnarray*}
$Y(n) = \xi(n)$, $\hat{p}_y(j|i) = p^\alpha_{ij}(x)$. Also, $\iota(n)$ is a vector whose $i$th component is
\begin{equation}
\varepsilon(n)(m_i^{-1} - \tilde{p}^\alpha_{\xi(n)i}(x(n))) + (\tilde{p}^\alpha_{\xi(n)i}(x(n)) - p^\alpha_{\xi(n)i}(x(n))) \to 0. \label{iota}
\end{equation}
(The presence of $\iota(n)$ does not affect the convergence, see the third `extension' in section 2.2, \cite{BorkarBook} which applies to the stochastic approximation with Markov noise as well.)
The stochastic matrix $[[p_{ij}^\alpha(x)]]_{i,j\in\V}$  is parametrized by the probability vector $x \in \SA_m$.\ For fixed $x$, let $\pi^\alpha(x)$ denote its stationary distribution, whose existence and uniqueness is ensured for each fixed $x \in$ int$(\SA_m)$ by our irreducibility assumption for $\G$
(see e.g., \cite[Section~6.1]{AFH2013}). A direct calculation shows that
$$\tilde{\pi}_i^\alpha(x) := \frac{f_i^\alpha(x)\sum_{k \in \N(i)}f^\alpha_k(x)}{\sum_\ell( f^\alpha_\ell(x)\sum_{k\in\N(\ell)}f^\alpha_k(x))}, \ i \in \V,$$
satisfies the local balance condition
$\tilde{\pi}^\alpha_i(x)p^\alpha_{ij}(x) = \tilde{\pi}^\alpha_j(x)p^\alpha_{ji}(x),$
because both sides equal
$$\frac{f_i^\alpha(x)f^\alpha_j(x)\mathbb{I} \big\{j \in \N(i)\big\} }{\sum_\ell( f^\alpha_\ell(x)\sum_{k\in\N(\ell)}f^\alpha_k(x))},$$
where $\mathbb{I} \big\{j \in \N(i)\big\} = \mathbb{I} \big\{i \in \N(j)\big\}$.\ So $\pi^\alpha(x) = \tilde{\pi}^\alpha(x)$.
%

We apply the `o.d.e.\ approach'  to our problem. Thus let $\varphi^\alpha_i(x) := f_i^\alpha(x)\sum_{j\in\N(i)}f_j^\alpha(x)/x_i$ and consider the o.d.e.
\begin{equation}
\dot{x}_i(t) = \frac{x_i(t)\varphi_i^\alpha(x(t))}{\sum_kx_k(t)\varphi^\alpha_k(x(t))} - x_i(t).\ \label{scaled-rep}
\end{equation}
Note that every equilibrium of (\ref{scaled-rep}) satisfies the fixed point equation
\begin{equation}
\pi(i) = h_i(\pi) := \frac{f_i^\alpha(\pi)\sum_{j \in \N(i)}f_j^\alpha(\pi)}{\sum_kf_k^\alpha(\pi)\sum_{\ell \in \N(k)}f_\ell^\alpha(\pi)} \ \ \forall i. \label{fixedpoint}
\end{equation}
Set $h(\cdot) := [h_1(\cdot, \cdots , h_m(\cdot)]$. By irreducibility, every such $\pi$ must be in int$(\SA_m)$.

\begin{lemma} The o.d.e.\  (\ref{scaled-rep}) has the same trajectories and the same asymptotic behavior as
the o.d.e.\
\begin{equation}
\dot{z}_i(t) = z_i(t)\left(\varphi_i^\alpha(z(t)) - \sum_jz_j(t)\varphi_j^\alpha(z(t))\right),\ \label{ode-rep}
\end{equation}
i.e., $z(t) = x(\tau(t))$ for some $t \in [0, \infty) \mapsto \tau(t) \in [0, \infty)$ which is strictly increasing and satisfies $t\uparrow\infty \Longleftrightarrow \tau(t)\uparrow\infty$.
\end{lemma}

\begin{proof} Since the r.h.s.\ of (\ref{ode-rep})  is locally Lipschitz in the interior of $\SA_m$, (\ref{ode-rep}) has a unique solution when $z(0) \in$ int$(\SA_m)$.\ We obtained  (\ref{ode-rep}) from (\ref{scaled-rep})  by multiplying the r.h.s.\ of (\ref{scaled-rep}) by the positive scalar valued bounded function $q(t) := \sum_kx_k(t)\varphi^\alpha_k(x(t))$, which is bounded away from zero uniformly in $t$.\  This amounts to a pure time scaling $t \mapsto \tau(t)$ where $\tau(\cdot)$ is specified by the well-posed differential equation $\dot{\tau}(t) = q(\tau(t))$.\ Then $z(t) := x(\tau(t))$.\ (The same device was  used in \cite{Benaim}, p.\  368.) Also, for suitable $\infty > c_2 > c_1 > 0$, $c_1t \leq \tau(t) \leq c_2t$.\  In particular, $\tau(t)\uparrow\infty$ as $t\uparrow\infty$, so the entire trajectory is covered.\ The claim follows.\ \end{proof}

 The dynamics (\ref{ode-rep}) is a special case of \textit{replicator dynamics}  \cite{Sandholm} (as is equation (3), \cite{Benaim}, p.\  368, in a similar context).\  Note also that an  equilibrium $z^*$ of (\ref{ode-rep}) must satisfy
\begin{equation}
z^*_i > 0 \ \Longrightarrow \ \varphi^\alpha_i(z^*) = \sum_jz^*_j\varphi^\alpha_j(z^*).\ \label{equil}
\end{equation}
In particular, $\varphi^\alpha_i(z^*) \equiv$ a constant for $i \in$ the support of $z^*$.

Let $A := [[a_{ij}]]_{i,j\in\V}$ be the (symmetric) adjacency matrix of $\G$.
Then for $x = [x_1, \cdots , x_m] \in \SA_m,$
$$\varphi_i^\alpha(x) = \frac{\partial}{\partial x_i}\Psi^\alpha(x) \ \mbox{for} \ \Psi^\alpha(x) := \frac{1}{2\alpha}\sum_{i,j}a_{ij}f_i^\alpha(x)f_j^\alpha(x).$$

Thus (\ref{ode-rep}) corresponds to the replicator dynamics for a potential game with potential $-\Psi^\alpha$ \cite{Sandholm}.
In what follows, by \textit{local maximum} of a function we mean a point in its domain where a local maximum is attained and  not the function value there.\  We make the following assumption which is generically true  (i.e., true for almost all parameter values, see, e.g., \cite{Matsumoto}, Chapter 2).\\

\noindent \textbf{(A2)} \textit{The equilibrium points of (\ref{scaled-rep}) (i.e., the fixed points of (\ref{fixedpoint})) are isolated and hyperbolic, i.e.,  the Jacobian matrix of $h$ at these points does not have eigenvalues on the imaginary axis. Also, their stable and unstable manifolds, which exist by hyperbolicity, intersect transversally if they do.}\footnote{This makes it a special case of a `Morse-Smale system'.}\\

In view of the preceding discussion, this amounts to the requirement that the Hessian of $\Psi^\alpha$ be nonsingular at its critical points in int$(\SA_m)$.

%

\begin{theorem}\label{replconvergence} For each $\alpha > 0$, the local maxima of $\Psi^\alpha : \SA_m \mapsto \mathbb{R}$ are stable equilibria of  (\ref{scaled-rep})  and the iterates of (\ref{SA})  converge to the set thereof, a.s.\  \end{theorem}

\begin{proof} Since (\ref{scaled-rep}) and (\ref{ode-rep}) are obtained from each other by a time scaling $t \mapsto \tau(t)$ that satisfies $\tau(t) = \Theta(t)$, it suffices to consider only (\ref{ode-rep}).\ We have
\begin{eqnarray}
\lefteqn{\frac{d}{dt}\Psi^\alpha(z(t))  }\nonumber \\
&=& \sum_iz_i(t)\left(\varphi^\alpha_i(z(t)) - \sum_jz_j(t)\varphi_j^\alpha(z_j(t))\right)^2  \nonumber \\
&\geq& \  0.  \label{Liap}
\end{eqnarray}
Thus $-\Psi^\alpha$ serves as a Lyapunov function for (\ref{ode-rep}), implying that it converges to the set of  critical and Kuhn-Tucker points of $\Psi^\alpha$.\ The local maxima will then correspond to stable equilibria.\ We next argue that the iterates converge to some local maximum a.s.\  By Corollary 8, p.\ 74, \cite{BorkarBook}, for stochastic approximation with Markov noise, combined with the first bullet of section 2.2, p.\ 16,  and Corollary 4, p.\  18, \cite{BorkarBook} (both of which which work with Markov noise for exactly identical reasons) and (A2),  the iterates converge a.s.\  to  a single, possibly sample path dependent,  critical or Kuhn-Tucker point of $\Psi^\alpha$.\  That it  must be a stable equilibrium,  i.e., a local maximum, follows by a variant of the theory developed in section 4.3, pp.\  40-47, \cite{BorkarBook}.\ This argument is very technical and is sketched in Appendix II.\ \end{proof}

The next lemma is similar to Theorem 6.3 of \cite{Benaim}, see also Theorem  5.1 of \cite{Arthur2}, reproduced as Chapter 10 of \cite{Arthur}. We sketch a brief proof for the sake of completeness.

\begin{lemma}\label{localmax} The probability of convergence of $\{x(n)\}$ in (\ref{SA}) to any local maximum of $\Psi^{\alpha}$ in $\SA_m$ is strictly positive.\ \end{lemma}

\begin{proof} Let $x^*$ be a local maximum and $O$ its domain of attraction for  (\ref{scaled-rep}).  Since the graph is  irreducible and the probability of next choice being $j$ is strictly positive $\forall \ j \in \N(i)$ when the current choice is $i$, it follows that the probability of $\{x(n)\}$ reaching $O$ from  any initial condition in finitely many steps is strictly positive. Once in $O$, the probability of convergence to $x^*$ is strictly positive by Theorem III.4 of \cite{Karmakar}, implying the claim. \end{proof}

We have
$$\Psi^\alpha(\pi) = \frac{1}{2\alpha}\sum_{i,j}(\mu_i\mu_j)^\alpha a_{ij}\pi_i^\alpha\pi_j^\alpha.$$

\begin{corollary}\label{optprob} The local maxima of $\Psi^\alpha$ are of the form $\pi(i) = z(i)^\frac{1}{\alpha}$ where $z$ is a local maximum of the quadratic form in $\{x_i\}$ given by $\sum_{i,j}x_ix_j(\mu_i\mu_j)^\alpha a_{ij}$, over the set
$$B^\alpha := \{y : y(i) \geq 0 \ \forall i, \ \sum_iy(i)^{\frac{1}{\alpha}} = 1\}.$$
\end{corollary}

\section{`Annealed' dynamics}\label{anneal}

In this section, we consider the `annealed' dynamics. That is,
taking a cue from simulated annealing \cite{Hajek}, we  consider the asymptotics as $\alpha\uparrow \infty$, corresponding to the `temperature' $T := 1/\alpha \downarrow 0$, slowly with time.\ A behavioral interpretation is that the agents exhibit a herd behavior, weighing in public opinion more and more with time.
We first analyze the optimization problem described in Corollary \ref{optprob} as $\alpha\uparrow\infty$.\ 
The set of limit points of $B^\alpha$ as $\alpha\uparrow\infty$ is given by (see Fig.\ \ref{collapse})
$B^\infty := \cap_{\alpha > 0}\overline{(\cup_{\alpha' > \alpha}B^{\alpha'})}$
$\supset B^* :=  \{e_i, 1 \leq i \leq m\},$
where $e_i, 1 \leq i \leq m,$ are the unit coordinate vectors.\  
Let
\begin{equation}
\label{D}
D := \{i \in \V : \mu_i = \max_j\mu_j\}
\end{equation}
and $\Pi^\alpha := \{\pi \in \SA_m : \pi$ is a local maximum of $\Psi^\alpha\}, \alpha >  0$.

\begin{figure}\
      \includegraphics[width=72mm,height=60mm]{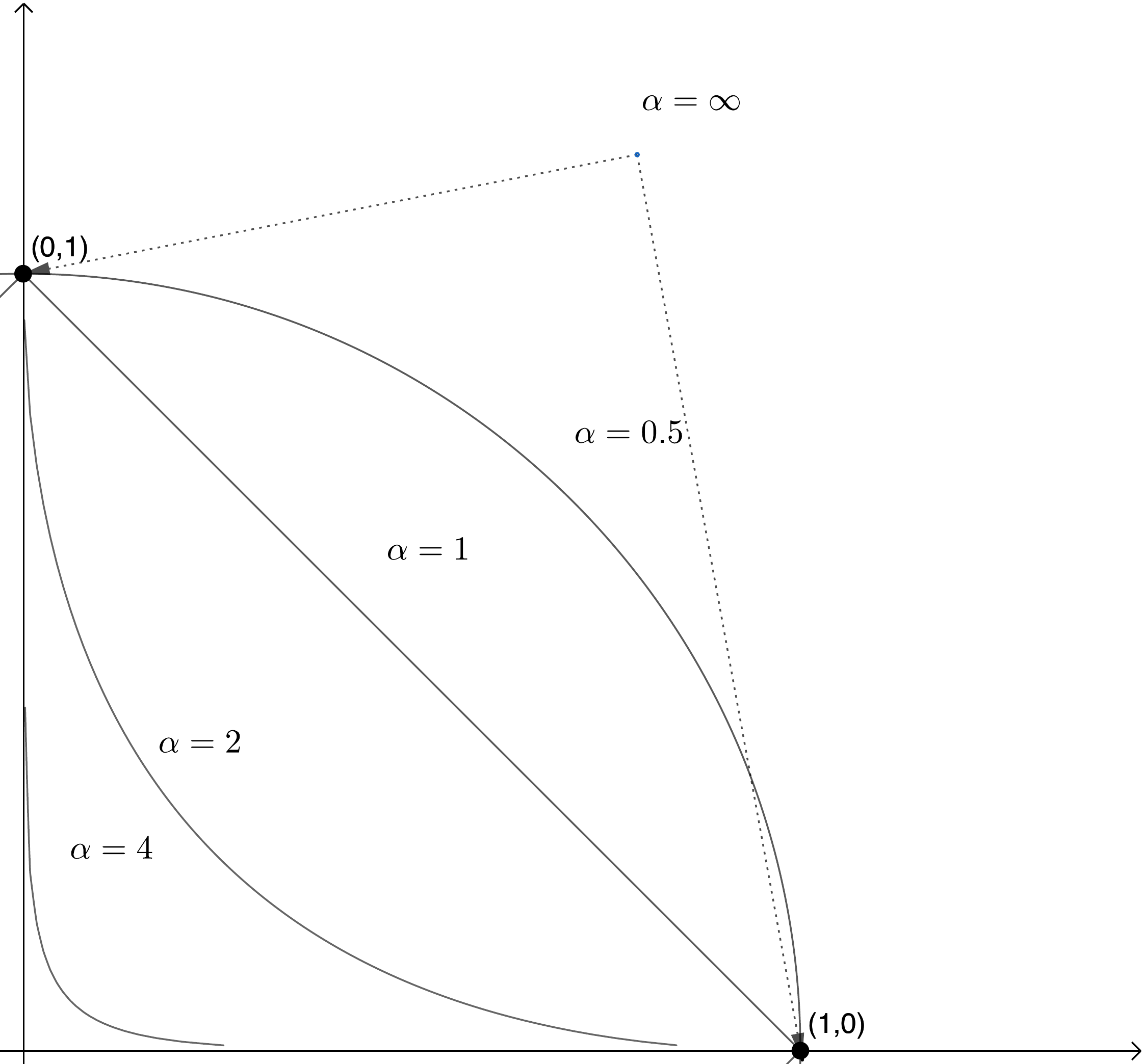}
      \caption{An illustration of the collapse of sets $B^\alpha$ to $B^\infty$.\ }
      \label{collapse}
 \end{figure}

\begin{lemma}\label{pointed} If $\alpha_n\uparrow\infty$ and $\pi_n \in \Pi^{\alpha_n}$, then $\pi_n \to B^*$.  \end{lemma}

\begin{proof} We are concerned here only about the relative sizes (i.e., ratios) of the summands in the definition of $\Psi^\alpha$. So we may assume that $\max_i\mu_i = 1$ and drop the factor $\frac{1}{2\alpha}$ in the definition of $\Psi^\alpha$.\   This simplifies the analysis while not affecting the location of local maxima and the relative magnitudes of the function values there.\ Let $S^* := \{i : \mu_i = 1\}$. Then
$\sum_{i,j}(\mu_i\mu_j)^\alpha a_{ij}x(i)x(j) \stackrel{\alpha\uparrow\infty}{\rightarrow} 0$
uniformly outside any relatively open neighborhood of $B^*$ in $\SA_m$. Hence
\begin{eqnarray*}
\lefteqn{\max_{x \in \SA_m}\sum_{i,j}(\mu_i\mu_j)^\alpha a_{ij}x(i)x(j)} \\
&\stackrel{\alpha\uparrow\infty}{\rightarrow}& \max_{x\in B^\infty}\sum_{i,j}(\mu_i\mu_j)^\alpha a_{ij}x(i)x(j) =1,
\end{eqnarray*}
which is attained at some $e_i, i \in S^*$. The claim follows. \end{proof}

\bigskip

Recall from (\ref{fixedpoint}) that $\pi^\alpha$ is a (not necessarily unique) solution to the fixed point equation
\begin{eqnarray}
\pi^\alpha (i) &:=& \frac{f_i^\alpha(\pi^\alpha)\sum_{k \in \N(i)}f^\alpha_k(\pi^\alpha)}{\sum_\ell(f^\alpha_\ell(\pi^\alpha)\sum_{k\in\N(\ell)}f^\alpha_k(\pi^\alpha))}. \   \ \label{FP0}
\end{eqnarray}
Decrease $T := 1/\alpha$ slowly according to the iteration
\begin{equation}
T(n+1) = (1 - b(n))T(n), \ n \geq 0, \label{cooling}
\end{equation}
where $1 > b(n) \downarrow 0$ are stepsizes satisfying
\begin{equation}
\sum_nb(n) = \infty, \ nb(n) \stackrel{n\uparrow\infty}{\rightarrow} 0, \ b(n) = o(c(n)).\ \label{cool}
\end{equation}
The second condition implies $\sum_nb(n)^2 < \infty$.
Assume that $x(0) \in$ int$(\SA_m)$.\ This is not a restriction, since $x(n) \in$ int$(\SA_m)$  from some $n$ on when all possible choices have been made at least once and the above requirement can be ensured simply by counting time from then on.\ Our main result is the following, reminiscent of `\textit{stochastically stable}' equilibria of \cite{Young2}.

\begin{theorem}\label{main} $\sum_{i\in D}x_i(n) \to  1$ a.s.
\end{theorem}

\begin{proof} The second and third conditions in (\ref{cool}) render the pair (\ref{SA}), (\ref{cooling}) a two time scale stochastic approximation with (\ref{SA}) run on a fast time scale and  (\ref{cooling}) run on a slower time scale. In fact the situation is simpler than the general two time scale schemes because the latter does not depend on the former, the dependence is unidirectional. We shall use the results of \cite{Yaji}. In \cite{Yaji}, stochastic recursive \textit{inclusions } involving set-valued maps on both time scales are considered. In (\ref{SA}),  (\ref{cooling}), we have instead single valued Lipschitz maps  for which assumptions A1-A8 of \cite{Yaji} are  easily verified. Our slow iteration (\ref{cooling}) has a unique limit $0$, whence A10 of \cite{Yaji} is trivially satisfied. This leaves the verification of assumption A9 of \cite{Yaji}. Consider (\ref{SA}) for fixed $\alpha = 1/T, \varepsilon(n) \equiv 0,$ and define:
$$D_0^{T} := \{\pi : \pi \ \mbox{satisfies the fixed point equation (\ref{FP0})}\}.$$
Let $D^{T} :=$ the closed convex hull of $D_0^{T}$.
 Using the fact that $T(n)$ update on a slower time scale and hence are `quasi-static' for the faster time scale of $x(n)$ (cf.\ the `two time scale' methodology of \cite{BorkarBook},  secion 6.1), we first `freeze' the slow components $T(n) \approx T$ and analyze the fast iterate (\ref{SA}). By the theory of stochastic approximation with Markov noise (see \cite{BorkarBook}, Chapter 6), it tracks the o.d.e.\  (\ref{scaled-rep}), a time-scaled version of  (\ref{ode-rep}) as observed earlier.\ Thus it converges to  $D^{T}$ by Theorem \ref{replconvergence}. We next show that as $ T = T(n)\downarrow 0$ and $\tilde{\pi}_n \in D^{T(n)} \ n \geq 1,$ $\tilde{\pi}_n \to $ the set $D$ defined in (\ref{D}).\ Consider a subsequence $\widetilde{T}(n) \downarrow 0$ such that
$$\tilde{\pi}_n := \pi^{\alpha}\Big|_{\alpha = 1/\widetilde{T}(n)} \to \pi^*$$
for some $\pi^*\in \SA_m$ with support $S^*$.\  Rewrite (\ref{FP0}) as
\begin{eqnarray*}
\tilde{\pi}_n(i)
&=& \frac{\sum_{j\in\N(i)}[\mu_i\mu_j\tilde{\pi}_n(i)\tilde{\pi}_n(j)]^{1/\tilde{T}(n)}}{\sum_{i'}\sum_{j\in\N(i')}[\mu_{i'}\mu_j\tilde{\pi}_n(i')\tilde{\pi}_n(j)]^{1/\tilde{T}(n)}} \\
&=& \frac{\frac{\sum_{j\in\N(i)}[(\mu_i\mu_j\tilde{\pi}_n(i)\tilde{\pi}_n(j)]^{1/\tilde{T}(n)}}{\max_{k,l\in \N(k)}[\mu_k\mu_l\tilde{\pi}_n(k)\tilde{\pi}_n(l)]^{1/\tilde{T}(n)}}}{\sum_{i'}\frac{\sum_{j\in\N(i')}[\mu_{i'}\mu_j\tilde{\pi}_n(i')\tilde{\pi}_n(j)]^{1/\tilde{T}(n)}}{\max_{k,l\in\N(k)}[\mu_k\mu_l\tilde\in_n(k)\tilde{\pi}_n(l)]^{1/\tilde{T}(n)}}}.
\end{eqnarray*}
As $\tilde{T}(n)\downarrow 0$, this concentrates on the set of $(i,j) \in \E$ for which
\begin{eqnarray*}
\lefteqn{\mu_i\pi^*(i)\sum_{j\in  \N(i)\cap S^*}\mu_j\pi^*(j)} \\
&& = \  \max_k\left(\mu_k\pi^*(k)\sum_{\ell\in  \N(k)\cap S^*}\mu_\ell\pi^*(\ell)\right).
\end{eqnarray*}
Combined with Lemma \ref{pointed}, this implies that the measure will concentrate on the $i$ such that
$$\mu(i)^2 = \max_j\mu(j)^2,$$
i.e., on $D$.\ Setting $D^{1/T} = D$ when $T = 0$, this verifies A9 of \cite{Yaji} for our purposes\footnote{It is also clear that the limiting measure will be uniform on $D$.}. Then Theorem 4, p.\ 1435, \cite{Yaji}, holds. We note that in the notation of this theorem, $\mathscr{Y} = \{0\}$ and $\lambda(y) = D^{1/y}$, whence the claim follows.
\end{proof}

\bigskip

\begin{figure*}
  \begin{minipage}{0.5\linewidth }
    \centering\includegraphics[width=0.9\linewidth]{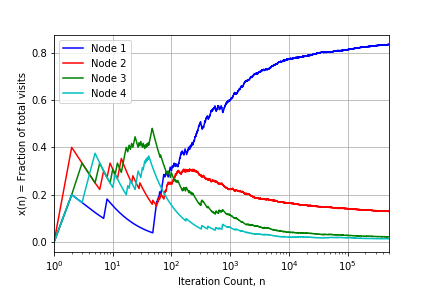}
   \subcaption{Linear Topology}
   \end{minipage}%
   \begin{minipage}{0.5\linewidth}
    \centering\includegraphics[width=0.9\linewidth]{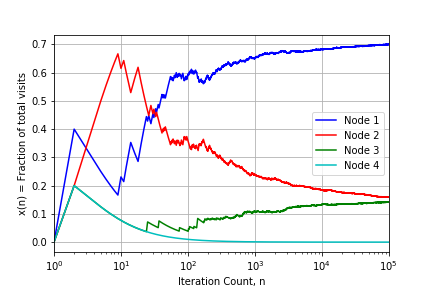}
     \subcaption{Star Topology}
	 \end{minipage}%
  \caption{ Fraction of total Visits, $x(n)$ Vs.\ Iteration Count for Linear and Star Topology.\ }
\end{figure*}

\begin{figure*}
  \begin{minipage}{0.5\linewidth }
    \centering\includegraphics[width=0.9\linewidth]{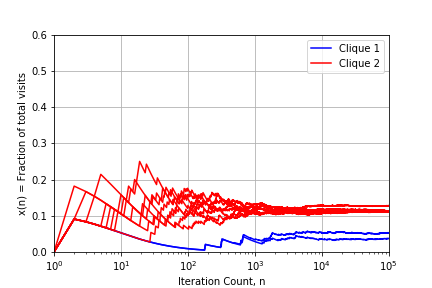}
   \subcaption*{(a) Initialize in clique-2, $\alpha$ fixed.\ }
   \end{minipage}%
   \begin{minipage}{0.5\linewidth}
    \centering\includegraphics[width=0.9\linewidth]{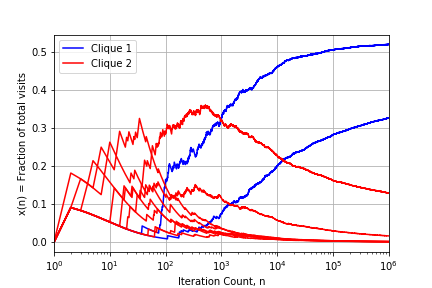}
    \subcaption*{(b) Initialize in clique-2, $\alpha \to \infty$.}
      \end{minipage}%
  \caption{ Fraction of total Visits, $x(n)$ Vs.\ Iteration Count for the two clique experiment }
\end{figure*}

\section{The unconstrainted case}\label{fullyconnected}

In this section we consider the case without graphical constraints, i.e., when the graph $\G$ is fully connected, where we can say more.
The case without graphical constraints can be viewed as a special case with $\G =$ the complete graph, i.e., $a_{ij} = 1 \ \forall \ i,j$.\ Then $\Psi^\alpha(x) = \left(\sum_if^\alpha_i(x)\right)^2$, which is convex for $\alpha \geq 1$, where the absence of graphical constraints does allow us to make stronger statements. Unfortunately this does not buy us stronger results for the $\alpha\uparrow\infty$ asymptotics. However, the story is different for a fixed $\alpha \in (0,1)$, where we indeed can say much more than in the graphically constrained case. Specifically, we get desired convergence guarantees even for a fixed $\alpha$ in this range, and make an analogy with  Ant Colony Optimization \cite{Ammar, BorkarDas}.

For $\alpha \in (0,1)$, since the expression being squared is non-negative, we can equivalently consider the problem of maximizing $\psi^\alpha(x) := \sum_if^\alpha_i(x)$, which is \textit{strictly concave}.\ Hence it has a unique maximum on $\SA_m$ to which our scheme will converge even without annealing.\ In fact, in this case,
the stationary solution
can be specified explicitly using the Lagrange multiplier technique as:
\begin{equation}
\label{xoptconc}
x_i(\infty) = \frac{\mu_i^{\alpha/(1-\alpha)}}{\sum_{k=1}^{m}\mu_k^{\alpha/(1-\alpha)}} \ .
\end{equation}
From (\ref{xoptconc}), as $\alpha \to 1$, the frequencies $x_i(\infty)$ start to
concentrate on  $D$ defined in (\ref{D}).
As seen in the simulation section, in practice one does not
need to take $\alpha$ very close to one.
If $\alpha=1$,
the replicator dynamics has the well studied linear payoffs and converges to a solution with only one nonzero component by standard arguments.

Now consider the case of $\alpha > 1$ with $\varepsilon(n) \equiv$ a constant $\varepsilon > 0$.
Note that in the unconstrained case, given $x$, the transition probability matrix $[[p^{\alpha,\varepsilon}_{ij}(x)]]$
is a stationary probability matrix with the identical rows $\pi^{\alpha,\varepsilon}(x)$ given by
$$
\pi^{\alpha,\varepsilon}_i(x) := (1-\varepsilon)\frac{f_i^{\alpha}(x)}{\sum_kf_k^{\alpha}(x)} + \varepsilon\frac{1}{m}.
$$
Hence its stationary distribution coincides with its (identical) rows. By Corollary 8, p.\ 74, \cite{BorkarBook},  the sequence $\{x(n)\}$ tracks the o.d.e.
\begin{equation}
\label{epsdyn}
\dot{x}_i(t) = \pi^{\alpha,\varepsilon}(x(t))  - x_i(t),
\end{equation}
i.e.,
$$
\dot{x}_i(t) = (1-\varepsilon)\frac{f_i^{\alpha}(x(t))}{\sum_kf_k^{\alpha}(x(t))} + \varepsilon\frac{1}{m}  - x_i(t).
$$
The stationarity condition for the above o.d.e.\ gives
\begin{equation}
(1-\varepsilon)\frac{f_i^{\alpha}(x)}{\sum_kf_k^{\alpha}(x)} + \varepsilon\frac{1}{m} - x_i = 0 \ \ \forall i. \label{fixpt}
\end{equation}
If $\varepsilon \to 1$, then by standard continuity arguments, $x \to$ the  set of solutions to (\ref{fixpt}) corresponding to $\varepsilon = 1$. This is a singleton consisting of the uniform distribution $x_i = \frac{1}{m} \ \forall i$. The map
\begin{eqnarray*}
\lefteqn{(x,\varepsilon) \mapsto F(x,\varepsilon) :=}\\
&&  (1-\varepsilon)(\sum_kf_k^\alpha(x))^{-1}[f^\alpha_1(x), \cdots , f^\alpha_m]  + \frac{\varepsilon}{m}I- x
\end{eqnarray*}
has a nonsingular Jacobian matrix $-I$ w.r.t.\ $x$ in int$(\SA_m)$ at $\varepsilon = 1$. Hence by the implicit function theorem, the fixed point $x^\varepsilon$ of (\ref{fixpt}) is an analytic function in a small neighborhood of the uniform distribution \cite{AFH2013}, i.e.,
$$
x_i(\varepsilon) = \frac{1}{m} + (1-\varepsilon) x^{(1)}_i + ... \ .
$$
Substituting this expansion in the stationarity condition (\ref{fixpt}) and equating terms with the
same powers of $1 - \varepsilon$ yields
$$
x^{(1)}_i = \frac{\mu^\alpha_i}{\sum_{k=1}^{m}\mu^\alpha_k}-\frac{1}{m}.
$$
This implies that the states with indices in the set $D$ will obtain a larger fraction
of visits in comparison with the other states.\ This is reminiscent of the Ant Colony Optimization algorithm of \cite{Ammar, BorkarDas} where
the initial randomness itself builds up the bias in favor of the optimum, to which the scheme converges \textit{with high probability}. A very fine analysis of the $\alpha > 1$ case for a related model appears in \cite{BenaimRS}.


The payoff functions $\{\varphi_i^\alpha(\cdot)\}$ in (\ref{ode-rep}) are of the form $\varphi^\alpha_i(z) = g_i(z_i)h(z)$ for  $h(\cdot) : \SA_m \mapsto (0,\infty)$ and $g_i : [0,1] \mapsto \mathbb{R}^+$, where the latter are monotone increasing.\ As shown in Lemma 4, p.\  14,  \cite{BorkarDas},  corners of $\SA_m$, i.e.,  $\{e_i\}$, are stable equilibria for (\ref{ode-rep}) and the only ones to be so.\ Moreover, the domain of attraction of $e_i$ is  $\{z \in \SA_m: z_i > z_j, \ j \neq i\}$.\ In view of the foregoing, this makes it clear how the bias for the optimum builds up starting from a uniform prior.


\section{Simulation experiments}\label{Simulation}

In this section we empirically demonstrate our theoretical results on a star and linear graph topology (with $m=4$, see Fig.~2 and 3).\ For the linear topology, $\mu = (2, \frac{1}{4},\frac{1}{2},1)$, designed so as to demonstrate the hill descending capabilities (i.e.\ jump out of the local maximum at node 4) of the algorithm.\ The noise $\zeta_i(\cdot)$ is assumed to be $N(0, 0.1)$. The random exploration parameter is set as $\varepsilon(n) := \frac{1}{\log (n+1)}$.  As can be seen in Fig.~2, $x_1(n)$ (the fraction of visits to the node with the highest $\mu$) converges to 1 as $n\uparrow\infty$.\ We remark here that the cooling schedule $\{\alpha(n)\}$ is the most important (and sensitive) parameter of the algorithm.\ A too fast or constant cooling schedule may tend to make the algorithm get stuck in the local maximum at node 4.\ The cooling schedule we used was
$
\alpha(n+1) = \alpha(n) \left( 1- \frac{1}{n \log n }  \right)^{-1}.
$
 For initial few iterations, we keep $\alpha=10^{-2}$ fixed to promote exploration.
For the star topology, $\mu = (1, \frac{1}{3},\frac{1}{3},\frac{1}{3})$.\ The cooling schedule was the same as before.\ Here, the central node, i.e.\ the node connected to all other nodes, is node 4.
For comparison purposes, we have also tried
$\mu = (2, \frac{1}{4},\frac{1}{2},1)$ with the fixed $\alpha=0.85<1$ in the complete graph setting.
The dynamics always converges to the
stationary solution $(0.98, 0.000, 0.000, 0.019)$.\ This demonstrates
our conclusion from Section~\ref{fullyconnected} that in the unconstrained case for the values of $\alpha < 1$ even
not so close to one, a very significant portion of the mass
is concentrated on the optimal node.

Our next numerical experiment is aimed at highlighting the importance of annealing  for convergence of $x(n)$ to $D$.\ We consider a graph composed of two cliques connected through a single edge.\ The number of nodes for clique-1 is 2 and those for clique-2 is 8.\ We set the noise $\zeta_i =0$ for all $i$ for this experiment. The results have been plotted in
Fig.~3.\ We set $\mu_i =1$ for $i \in$ clique-1 and $\mu_i =0.5$  for $i \in$ clique-2.\ Some points to note are:
 \begin{itemize}
\item  If we initialize the walk in clique-2 and \textit{do not} increase $\alpha \to \infty$, then the relative frequencies converge to non-zero values for nodes in clique-2.\ (In Fig.~3(a), we have set $T=0.1 \,(\alpha=10)$.)

\item If we initialize the walk in clique-2 and \textit{do} increase $\alpha \to \infty$, then the chain moves to clique-1 and stays there.\ \

\end{itemize}

 With linear topology, we  make an important comparison with  the multiarmed bandit literature.  With nodes labeled $\{1,2,3,4\}$, the $\alpha\uparrow\infty$ limit corresponds to the transition probabilities
$$p(1|1), p(1|2), p(4|3), p(4|4) = 1, \ p(i|j) = 0 \ \mbox{otherwise}.$$
That is, the chain moves deterministically to the neighbor (including itself) with the highest reward. It has two communicating classes $\{1,2\}$ and $\{3,4\}$. For $\epsilon \in (0,1)$, the $\epsilon$-greedy policy has a stationary distribution that is seen to concentrate equally on $1,4$ as $\epsilon\downarrow0$ by the symmetry of the problem. In particular, it is a suboptimal distribution. A simple two time scale argument applied to (\ref{SA}) then shows that $x(n)$ converges this suboptimal distribution. In contrast, if we consider the corresponding fully connected graph with the same reward structure, the purely greedy policy given by the $\alpha\uparrow\infty$ limit has $p(1|i) = 1 \ \forall i$ and the stationary distribution is seen to concentrate on the optimal node $1$. In the fully connected case the $\varepsilon(n)$-greedy policy with $\varepsilon(n) = \frac{1}{n}$ converges to the optimal, as shown in Theorem 3 of \cite{Auer}. Thus, a standard bandit algorithm can fail in the graph-constrained framework.

In Fig.~4, we provide a comparison of the proposed algorithm with Simulated Annealing. We briefly describe the details of the modified version of SA we use here. The SA algorithm consists of a discrete time inhomogeneous Markov chain, whose transition mechanism $P(n):=[[p_{xy}(n)]]_{x,y\in \mathcal{V}}$ for temperature $T_n$ can be formally written as:	
 \begin{equation*}
  			 p_{x, y }(n) = \begin{cases}
   					 0, & \text{if $y \notin \mathcal{N}(x)$}\\
   					 \frac{1}{|\mathcal{N}(x)|} \exp\Big\{ \frac{- \big(\hat{\mu}_x(n) - \hat{\mu}_y(n)  \big)^+}{T_n} \Big\}, & \text{otherwise}
 					 \end{cases}
			\end{equation*}
   and
           $$
           p_{x,x}(n) = 1 - \sum_{i \in \mathcal{N}(x)}  p_{x, i }(n),
           $$
where $(x)^+:= \max(0,x)$ and $\hat{\mu}_x(n)$ is the empirical mean estimate at time $n$ of object $x$. To keep the comparison to our algorithm fair we update the empirical mean in the same manner as (\ref{slln}).

Judging from Fig. 4, our algorithm achieves a better medium and long run performance in terms of relative frequency of the optimal reward for both linear and star topology. The time step for SA is kept equal to $\frac{\gamma}{\log(1+k)}$, where $\gamma=0.1$ is selected empirically to give the best performance.

\begin{figure}[h]
\label{fig:SAcomp}
\includegraphics[width=8.5cm, height =6cm]{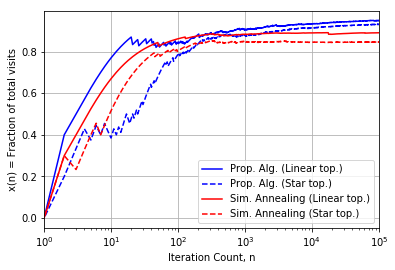}
\caption{An empirical comparison of SA with the proposed algorithm for star and linear topology. The reward vectors are kept the same as the previous experiments.}
\end{figure}

\bigskip

\noindent \textbf{Appendix I}\\

\noindent \textbf{Proof of Lemma 2.1 :}  This follows from the strong law of large numbers   if
\begin{equation}
S_i(n)  \uparrow \infty, \label{infinite}
\end{equation}
and our convergence analysis applies.\  But  (\ref{infinite}) follows from the fact $\sum_n\varepsilon(n) = \infty$, because by the conditional Borel-Cantelli lemma (Lemma 17, p.\  49, of \cite{BorkarBook}),
$$\sum_n\mathbb{I}\{\xi(n+1) = i\} = \infty \Longleftrightarrow \sum_n P(\xi(n+1) = i | \F_n) = \infty$$
a.s.\ Now $\chi(\xi(n))$ assigns mass $\frac{1}{|\N(i)|} \geq \frac{1}{m}$ to $i$ when $\xi(n) \in \N(i)$ and $0$ otherwise. Hence
$$
\sum_n P(\xi(n+1) = i | \F_n)
\ge  \frac{1}{m}\sum_{j \in \N(i)}\sum_n\varepsilon(n)I\{\xi(n) = j\},
$$
By the conditional Borel-Cantelli lemma,
\begin{eqnarray*}
&& \frac{1}{m}\sum_{j \in \N(i)}\sum_n\varepsilon(n)I\{\xi(n) = j\} = \infty \\
&\Longleftrightarrow& \frac{1}{m} \sum_{j \in \N(i)}\sum_n\varepsilon(n)P(\xi(n) = j|\F_{n-1}) = \infty.
\end{eqnarray*}
Using a similar bound for $P(\xi(n) = j|\F_{n-1})$ yields
\begin{eqnarray*}
&& \frac{1}{m}\sum_{j \in \N(i)}\sum_n\varepsilon(n)P(\xi(n) = j|\F_{n-1}) \\
&\geq&  \frac{1}{m^2}\sum_{k \in \N(j)}\sum_{j \in \N(i)}\sum_{n\geq 1}\varepsilon(n)^2I\{\xi(n-1) = k\},
\end{eqnarray*}
and so on, so  combining all these inequalities and using (\ref{varep0}),
$$
\frac{1}{m^m}\sum_{n \geq m}\varepsilon(n)^m = \infty
 \Longrightarrow \sum_n\mathbb{I}\{\xi(n+1) = i\} = \infty
$$
a.s. Thus (\ref{infinite}) holds.\\

\noindent \textbf{Appendix II}\\

Here we sketch the proof of the `avoidance of unstable equilibria a.s.' (also known as `avoidance of traps') result invoked in the proof of Theorem \ref{replconvergence}. This is based on the results of section 4.3, \cite{BorkarBook}, pp.\  44-51, originally from \cite{Borkar}. These in turn depend on the estimates of section 4.1, pp.\  31-41 of \cite{BorkarBook}.  We sketch the main steps, referring the reader to the above for details  common to both and highlight only the differences between the present set-up and that of section 4.3, \cite{BorkarBook}. For later reference, we use (A$n)^*, n \geq 1,$ to denote the assumptions of \textit{ibid.} and simply (A$n$) to refer to our own.

The proof of \textit{ibid.} is broadly in two parts. The bulk of the work is for the first part, which is to show that the iterates will  keep getting pushed away from the stable manifolds of unstable equilibria sufficiently often, a.s. This is an argument based on the conditional Borel-Cantelli lemma. In \cite{BorkarBook}, this argument relies on showing that the aggregated martingale noise over an interval approaches  a non-degenerate gaussian distribution under suitable scaling, by the central limit theorem for martingale arrays. This is ensured by assumption (A6)$^*$. The topological assumption (A5)$^*$ then ensures that there is enough probability of the iterates getting pushed away adequately and often enough that they move away from the manifold, to the domain of attraction of stable equilibria. The second part then says that it will converge to a stable equilibrium almost surely. This uses a concentration result from section 4.1 of \cite{BorkarBook}, which quantifies the probability of convergence to a stable equilibrium given that the current iterate is in its domain of equilibrium. For us, the second part simply amounts to replacing the latter result by its counterpart for Markov noise from \cite{Karmakar}. The first part is  what takes the most effort. While (A5)$^*$ can be ensured by imposing a reasonable assumption, (A6)$^*$ turns out to be more elusive, precisely because of graph constraints that imply motion only to neighboring nodes. Thus, the natural counterpart of (A6)$^*$ that would require the conditional covariance of $\xi(n+1)$ given $\F_n$ to be non-singular is simply false. Luckily, we need such non-singularity to hold in an average sense. Bulk of our work below will be towards establishing this. The condition (A7)$^*$ is simply replaced by its suitable counterpart here, so it is not a major issue.

It should also be added that the assumptions and proof of \cite{Borkar} followed here are among many such for `avoidance of traps' results, see \cite{Brandiere, Pemantle},  to name some others. Thus it seems eminently possible to adapt these to give alternative sets of assumptions and corresponding proofs  for Markov noise.

We begin by discussing the key assumptions (A5)$^*$-(A8)$^*$ in section 4.3, \cite{BorkarBook}, that are specific to the results therein. Assumptions (A1)$^*$-(A4)$^*$ of \textit{ibid.} are  generic assumptions for stochastic approximation that are already covered here.\ Let $m_i = |\N(i)|$.\
Define the $\{\F_n\}$-martingale difference sequence
\begin{eqnarray}
M_i(n+1)  &=& I\{\xi(n+1) = i\}   - (1 - \varepsilon(n))p^\alpha_{\xi(n)i}(x(n)) \nonumber \\
&&- \ \varepsilon(n)I\{i \in \N(\xi(n))\}/m_i. \label{ondu}
\end{eqnarray}
Let $a(n) := \frac{1}{n+1}, n \geq 0$.\ Then (\ref{SA}) can be written as
\begin{eqnarray}
x_i(n+1) &=& x_i(n) \ +
a(n)\Bigg[(1 - \varepsilon(n))p^\alpha_{x_i(n)j}(x(k)) \ +  \nonumber \\
&& \frac{\varepsilon(n)}{m_i}\Bigg] + \ a(n)M_i(n+1), \ 1 \leq i \leq m. \label{erdu}
\end{eqnarray}

Let $W$ denote the complement of the union of the domains of attraction of stable equilibria, i.e., the local maxima of $\Psi$.\ One important implication of (A2) is the following.\ Define the truncated open cone
$$C_\kappa := \left\{x \in \SA_m : 1 < x_1 < 2, \ \left|\sum_{i=2}^mx_i^2\right|^{1/2} < \kappa x_1\right\}$$
for some $\kappa > 0$. For any orthogonal matrix $O$, $x \in \mathbb{R}^d$ and $a>0$, we let $OD$, $x+D$  and $aD$ denote respectively, the rotation of $D$ by $O$, translation of $D$ by $x$, and scaling of $D$ by $a$. Then (A2) implies:\\

\noindent \textbf{(A2')} There exists $\kappa > 0$ such that for any $x \in \SA_m$ and sufficiently small $a > 0$, there exists an orthogonal matrix $O_{a,x}$ such that $B(x,a,\kappa) :=  x + aO_{x,a}C_\kappa$ satisfies: any $y \in B(x,a,\kappa)$ is at least distance $a$ away from $W$. \\

\begin{figure}\
      \includegraphics[width=90mm,height=60mm]{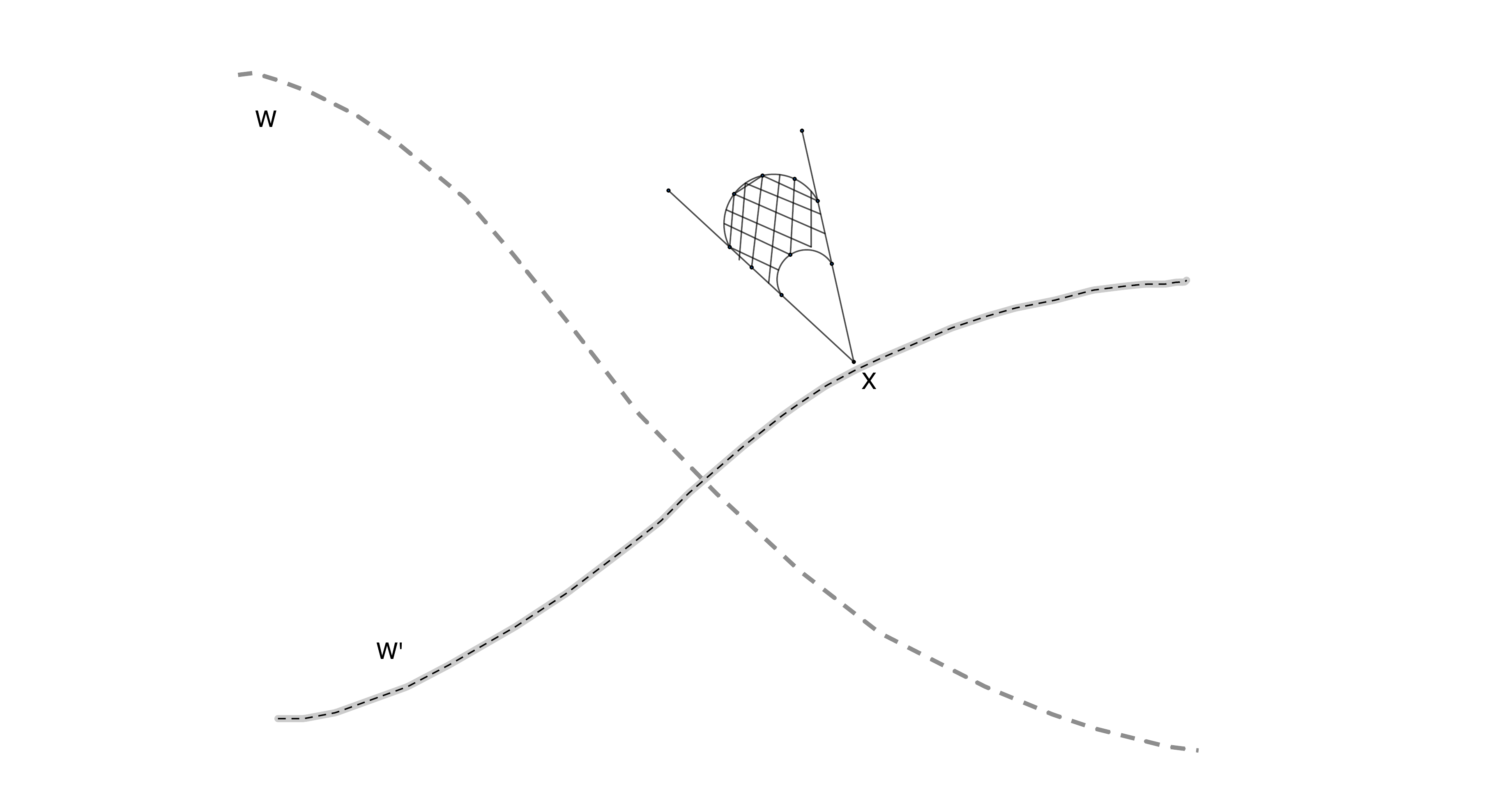}
      \caption{An illustration of Assumption 2'.\ }
      \label{A2'}
 \end{figure}

This means in particular that for any sufficiently small $a > 0$, we can plant a version of the truncated cone scaled down by $a$ near $x$ by means of suitable translation and rotation, in such a manner that it lies entirely in $W$. This ensures that any point in $\mathbb{R}^m$ cannot have points in the complement of $W$ arbitrarily close to it in all directions.\   This replaces (A5)$^*$.\  Next we consider (A6)$^*$. This is not appropriate for the `Markov noise' framework here, hence will have to be modified. We modify it by replacing $Q(x)$ there by $Q^n_i(x), 1 \leq i \leq m$, where $Q^n_{\xi(n)}(x(n))$ is the conditional covariance matrix of the random vector $[I\{\xi(n+1) = 1\}, \cdots , I\{\xi(n+1) = m\}]$ conditioned on $\xi(n), x(n)$, which is the same as `conditioned on $\F_n$' by virtue of conditional independence.\   Then $Q^n_i(x)$ has a $m_i\times m_i$ diagonal block $\bar{Q}^n_i(x)$, corresponding to rows and columns indexed by elements of $\N(i)$. Note also that $I\{\xi(n+1) = j\}$, $j \in \N(i)$, conditioned on $\xi(n), x(n)$, are conditionally Bernoulli random variables, albeit correlated.  The remaining rows and columns of $Q^n_i(x)$ are zero.  Thus $Q^n_i(x)$ is singular for each $i,n,$ and the obvious counterpart of (A6)$^*$,  which would require the least eigenvalue of the $Q^n_i(x)$'s to be bounded away from zero, is not tenable.\ However, a closer scrutiny of the arguments of section 4.3, \cite{BorkarBook}, specifically the last part of the proof of Lemma 16 there, shows that the actual requirement is weaker. We exploit this fact below.

An additional complication is that the smallest eigenvalue of the diagonal submatrices $\bar{Q}^n_i(x)$ is also zero because of the fact that $\sum_{j\in\N(i)}I\{\xi(n+1) = j\}= 1$ when $\xi(n) = i$  introduces degeneracy: the vector $\1 := [1, \cdots, 1]^T$ is always an eigenvector corresponding to eigenvalue $0$. However, our dynamics is confined to the probability simplex, a compact manifold with boundary, to which $\1$ is orthogonal. Thus we need to consider only the linear  transformations
\begin{eqnarray*}
&&y \in \R^{m_i} \mapsto \\
&&\SA^i := \{z \in \R^{m_i} : z_j \geq 0, 1 \leq  j \leq m_i, \sum_{j=1}^{m_i}z_j = 1\}.
\end{eqnarray*}
We show later that the least eigenvalue $\lambda_n(i)$ of $\bar{Q}^n_i(x)\Big|_{S^i}$ satisfies
\begin{equation}
\lambda_n(i) \geq  \frac{\varepsilon(n)}{m}, \label{ultimatebound}
\end{equation}
which in turn implies that
\begin{equation}
\bar{Q}_i(x)\Big|_{S^i} \geq \frac{\varepsilon(n)}{m} J_i\Big|_{S^i}, \label{compare}
\end{equation}
where $J_i :=$ the diagonal matrix with diagonal elements $= 1$ for rows and columns corresponding to $\N(i)$ and $= 0$ otherwise.
The inequality in (\ref{compare}) is w.r.t.\ the usual partial order for positive semidefinite matrices.
Denote by $I$ the $m$-dimensional identity matrix and by $D\pi^\alpha$ the Jacobian matrix of $\pi^\alpha$.
Also define
$$\varphi(n) = \Bigg(\sum_{k=n}^{s(n)}a(k)^2\varepsilon(k)\Bigg)^{\frac{1}{2}}.$$
where  $s(n) := \min\{k \geq n : \sum_{\ell = 1}^ka(k) \geq T\}$
for a prescribed $T > 0$. Then as in  p.\ 48, \cite{BorkarBook}, we have,
\begin{eqnarray}
\lefteqn{\frac{1}{\varphi(n)^2}\sum_{j=s(n)}^{s(n)+i-1}a(j)^2\times} \nonumber \\
&&\Bigg(\prod_{k=j+1}^{s(n)+i-1}(I + a(k)(D\pi^\alpha(x(n)) - I))\Bigg) \nonumber \\
&&\times Q_{\xi(n)}(x(n))\times  \nonumber \\
&&\Bigg(\prod_{k=j+1}^{s(n)+i-1}(I + a(k)(D\pi^\alpha(x(n)) - I))\Bigg)^T \nonumber \\
&&\geq \ \ \ \frac{1}{m\varphi(n)^2}\times  \nonumber
\end{eqnarray}
\begin{eqnarray}
\sum_{j=s(n)}^{s(n)+i-1}a(j)^2\Bigg(\prod_{k=j+1}^{s(n)+i-1}(I + a(k)(D\pi^\alpha(x(n)) - I))\Bigg)&&\nonumber \\
\times \varepsilon(k)J_{\xi(n)}\Bigg(\prod_{k=j+1}^{s(n)+i-1}(I + a(k)D\pi^\alpha(x(n)) - I))\Bigg)^T.&& \nonumber
\end{eqnarray}
\begin{equation}
\label{Bigeq}
\end{equation}
Define the random probability vector
$\nu(n) = [\nu_1(n), \cdots , \nu_s(n)]$ by
$$\nu_i(n) := \frac{\sum_{k=n}^{s(n)}a(k)^2\varepsilon(k)I\{\xi(k) = i\}}{\sum_{k=n}^{s(n)}a(k)^2\varepsilon(k)}$$
for $i \in S$.
Then an argument analogous to that of Lemma 6, pp.\  73-74, \cite{BorkarBook}, shows that a.s., every limit point $\pi^*$ of $\{\nu(n)\}$ is some stationary distribution $\pi^\alpha$ for $\{\xi(n)\}$. In particular, it has full support by virtue of (\ref{FP0}).\  By dropping to a further subsequence if necessary, consider a limit point of the r.h.s.\ of (\ref{Bigeq}). This will be of the form $\frac{1}{m}\int_0^t\Phi(T,s)(\sum_i\pi^*(i)J_i)\Phi(T,s)^Tds$ for some $t \geq 0$, where $\Phi( \cdot , \cdot )$ is the fundamental matrix for the linearization of the o.d.e.\  (\ref{scaled-rep}) restricted to $\SA_M$.\ This is clearly positive definite when restricted to $\SA_M$ $\big($because $\sum_i\pi^*(i)J_i$ is$\big)$.
The argument leading to Corollary 18 in \cite{BorkarBook}, pp.\  49, then goes through as before.

(A7)$^*$ is used in section 4.3, \cite{BorkarBook}, on p.\  50 alone.\ One key step in its application there is the use of the estimate of  trapping probability (i.e., the probability of convergence to a stable equilibrium conditioned on the iterates being in its domain of attraction), from Theorem 8, pp.\  37, \cite{BorkarBook}. This is used to conclude the proof in section 4.3 of \cite{BorkarBook}.\ That estimate cannot be used here because we are dealing with Markov noise.\ However, we can use the (stronger) concentration result from Theorem III.4, \cite{Karmakar} to conclude our desired result in a completely analogous manner. That said, we still need to verify, as in p.\ 50 of \cite{BorkarBook}, that

\begin{equation}
\sum_{k\geq n}\frac{1}{(k+1)^2} = o(\varphi(n)) =
 o\left(\sqrt{\sum_{k=n}^{s(n)}\left(\frac{\varepsilon(k)}{k+1}\right)^2}\right). \label{small-o}
\end{equation}
The l.h.s.\ is $\Theta\left(\frac{1}{n}\right)$. The r.h.s.\  is
$\Theta\left(\sqrt{\frac{T\varepsilon(s(n))^2}{s(n)}}\right) = \Theta\left(\frac{\varepsilon(n)}{\sqrt{n}}\right)$  because $s(n) = \Theta\left(ne^T\right)$.
Thus  (\ref{small-o})  amounts to $\frac{1/n}{\varepsilon(n)/\sqrt{n}} \to 0$, i.e.,
$\varepsilon(n) = \omega\left(\frac{1}{\sqrt{n}}\right)$.
This is the second condition in (\ref{varep}).

\medskip

(A8)$^*$ can be seen to hold in the interior of $\SA_m$, which is our state space of interest, because it follows from (\ref{FP0}) that the equilibria will be in the interior of $\SA_m$.

We have ignored the errors due to time variation of $\hat{\mu}_n, T(n)$ because they do  not affect the analysis. Both get multiplied by $a(n)$ and are therefore $o(a(n))$ in the `drift' (i.e., the driving vector field) of the algorithm and contribute only an asymptotically negligible error.  (See again the second bullet on p.\ 17 of \cite{BorkarBook} which applies to stochastic approximation with Markov noise as well.) The factor $a(n)\varepsilon(n)$ on the other hand multiplies the noise and therefore is what matters for `avoidance of traps'.\\

\noindent \textbf{Derivation of (\ref{ultimatebound}):}\\

For $\xi_n = i$,
\begin{eqnarray}
p_n(j) &:=& (1 - \varepsilon(n))p^\alpha_{ij}(x(n))I\{j \in \N(i)\} \nonumber \\
&&+ \ \varepsilon(n)I\{j \in \N(i)\}/m_i. \label{p-en}
\end{eqnarray}
Then  $p_n(j) \geq \frac{\varepsilon(n)}{m_i} \ \forall j \in \N(i)$. Fix $n$. Let $p = [p(1), \cdots , p(m_i)]$ be a probability vector in $S_0^i :=$ the simplex of probability vectors in $\R^{m_i}$ with each component $\geq \frac{\varepsilon(n)}{m_i}$ (in particular, $p_n(\cdot) \in S_0^i$).  Let $y = [y_1, \cdots , y_{m_i}]^T \in \R^{m_i}$ satisfy $\|y\|_2 = 1$ and $y\perp\1$ (i.e., $\sum_iy_i = 0$). Then
$$y^T\bar{Q}_i(x(n))y \geq \min_{p\in S_0^i}\left( \sum_{j\in\N(i)}p(j)y_j^2 - \left(\sum_j p(j)y_j\right)^2\right).$$
The function of $p(\cdot)$ in parentheses on the right is concave in $p(\cdot)$ for a fixed $x$ and will achieve its minimum at some corner of  $S_0^i$, say (without loss of generality) at
\begin{equation}
p := \left[1 - \frac{(m_i-1)\varepsilon(n)}{m_i}, \frac{\varepsilon(n)}{m_i}, \cdots , \frac{\varepsilon(n)}{m_i}\right]. \label{minprob}
\end{equation}
Then
\begin{eqnarray*}
y^T\bar{Q}_i(x(n))y &\geq& (1 - \varepsilon(n))y_1^2 + \frac{\varepsilon(n)}{m_i}\sum_iy_j^2 - \\
&&\left((1 - \varepsilon(n))y_1 + \frac{\varepsilon(n)}{m_i}\sum_iy_i\right)^2 \\
&=&  ((1 - \varepsilon(n)) - (1 - \varepsilon(n))^2)y_1^2 + \frac{\varepsilon(n)}{m_i} \\
&\geq& \frac{\varepsilon(n)}{m_i},
\end{eqnarray*}
where we use the identities $\sum_iy_i = 0, \ \sum_iy_i^2 = 1$. This completes the proof.\\

\noindent \textbf{Appendix III}\\

In this appendix, we provide an example of  $\{c(n)\}$ in (\ref{epreduce}).
Let $c(n) = \frac{1}{1 + (n+1)\log(n+1)}$ in (\ref{epreduce}). Then we have
 \begin{align*}
 \varepsilon(n) &= \prod_{k=1}^{n} \Big(1- \frac{1}{1 + (k+1) \log (k+1)}\Big) \varepsilon(0)\\
 & < \exp\Big( -\sum_{k=1}^{n}\frac{1}{1 + (k+1)\log (k+1)} \Big) \varepsilon(0)\\
 & <\exp \Big(- \log \log n \Big)\upsilon\varepsilon(0)\\
 & = \frac{\upsilon\varepsilon(0)}{\log n}
  \end{align*}
for some $\upsilon > 0$.  Thus $\varepsilon(n) = O\left(\frac{1}{\log n}\right)$. Next we show that $\varepsilon(n) = \Omega\left(\frac{1}{\log n}\right)$.  For this we use the fact for $x \in (0, 1)$,
  $$\log\left(\frac{1}{1-x}\right) \leq \frac{x}{1 - x} \Longrightarrow 1 - x \geq e^{-\frac{x}{1-x}}.$$
  Letting $\varepsilon(0) = 1$ without loss of generality,
  \begin{eqnarray*}
  \varepsilon(n) &=& \prod_{k=1}^{n}\left(1 - \frac{1}{1 + (k+1)\log(k+1)}\right) \\
  &\geq& \prod_{k=1}^n e^{-\frac{p_k}{1- p_k}} \ \mbox{for} \ p_k := \frac{1}{1 + (k+1)\log(k+1)} \\
  &=& e^{-\sum_{k=1}^n\frac{p_k}{1 - p_k}}.
  \end{eqnarray*}
As $p \downarrow 0$, $\frac{p}{1-p} = p(1 + o(1))$. Thus
$$\varepsilon(n) \geq e^{-\sum_{k=1}^n p_k(1 + o(1))}.$$
But
\begin{eqnarray*}
\sum_{k=1}^n p_k &\leq& p_1 + \int_0^n\frac{1}{1+(1+y)\log(1+y)}dy \\
&\leq& \log\log(n+1) + \log C'
\end{eqnarray*}
for suitable $C' > 0$. Hence for suitable $C > 0$,
\begin{eqnarray*}
\varepsilon(n) &\geq& Ce^{-(1 + o(1))\sum_{k=1}^n p_k}
\\
&\geq& Ce^{-(1 + \epsilon(n))\left(\log\log(n+1)\right)} \\
&&\mbox{where} \ \epsilon(n)\stackrel{n\uparrow \infty}{\to} 0,\\
&=& \frac{C}{(\log(n+2))^{1 + \epsilon(n)}}.
\end{eqnarray*}
That is, $\varepsilon(n) = \Theta((\log n)^{-1})$.

Using the above, it is easy to verify that $\{c(n)\}$ satisfies the stipulated conditions.\\

\end{document}